\documentclass[times,sort&compress,3p]{elsarticle}
\usepackage{amssymb}

\usepackage[USenglish]{babel} 
\usepackage{graphicx} 
\usepackage{tikz}
\usepackage{framed}
\usepackage{floatrow}
\usepackage{remreset} 
\usepackage[nouppercase]{scrpage2} 
\usepackage{amsmath} 
\usepackage{amssymb} 
\usepackage{faktor}
\usepackage[thmmarks, 
amsmath, 
hyperref 
]{ntheorem} 
\usepackage{bm} 
\usepackage{bbm} 
\usepackage{enumitem} 
\usepackage[hang]{subfigure} 
\usepackage{wrapfig} 
\usepackage{tabularx} 
\usepackage{color}
\usepackage{dcolumn} 
\usepackage{booktabs} 
\usepackage{listings} 
\usepackage{psfrag} 
\usepackage[pdftex, 
setpagesize=false, 
pdfborder={0 0 0}, 
pdfpagemode=UseOutlines, 
]{hyperref} 
\usepackage{algorithm}
\usepackage[noend]{algpseudocode}
\usepackage[labelfont=bf]{caption}

\theoremstyle{break} 
\theoremheaderfont{\bfseries}
\theorembodyfont{}
\theoremseparator{}
\newtheorem{definition}{Definition} 
\newtheorem{lemma}[definition]{Lemma}
\newtheorem{theorem}[definition]{Theorem}

\newtheorem{remark}{Remark} 
\newtheorem{example}{Example} 
\newtheorem{algorithm}{Algorithm}
\theoremstyle{nonumberplain} 
\theoremsymbol{$\Box$}
\newtheorem{proof}{Proof:}

\def\d{\mathrm{d}}

\newcommand*{\IP}{\Pr}
\newcommand*{\IE}{{\rm{E}}}
\newcommand*{\IR}{\mathbb{R}}

\newcommand*{\IN}{\mathbb{N}}

\newcommand{\norm}[1]{\left\lVert#1\right\rVert}

 \journal{...}

\begin{document}

\begin{frontmatter}

\title{Stochastic decomposition for $\ell_p$-norm symmetric survival functions on the positive orthant}

\author[A1]{Jan-Frederik~Mai\corref{mycorrespondingauthor}}
\author[A2]{Ruodu Wang}

\address[A1]{XAIA Investment GmbH, Sonnenstr.\ 19, 80331 M\"unchen}
\address[A2]{Department of Statistics and Actuarial Science, University of Waterloo}

\cortext[mycorrespondingauthor]{Corresponding author. Email address: \url{mai@tum.de}}


\begin{abstract}
We derive a stochastic representation for the probability distribution on the positive orthant $(0,\infty)^d$ whose association between components is minimal among all probability laws with $\ell_p$-norm symmetric survival functions. It is given by a transformation of a uniform distribution on the standard unit simplex that is multiplied with an independent finite mixture of certain beta distributions and an additional atom at unity. On the one hand, this implies an efficient simulation algorithm for arbitrary probability laws with $\ell_p$-norm symmetric survival function. On the other hand, this result is leveraged to construct an exact simulation algorithm for max-infinitely divisible probability distributions on the positive orthant whose exponent measure has $\ell_p$-norm symmetric survival function. Both applications generalize existing results for the case $p=1$ to the case of arbitrary $p \geq 1$.
\end{abstract}

\begin{keyword}
Archimedean copula \sep max-infinitely divisible \sep $d$-monotone function \sep simulation algorithm

\MSC[2020] 60E05 \sep 65C10 \sep 60E07

\end{keyword}

\end{frontmatter}

\section{Introduction}
We fix $p \geq 1$ and write $\theta:=1/p$ throughout to simplify notation. Let $\mu$ be a measure on $[0,\infty]^d$ with the property that its survival function takes the specific form
\begin{gather}
\mu\big( (\bm{x},\bm{\infty}]\big) = \varphi\big( \norm{\bm{x}}_p\big),\quad \bm{x}> \bm{0},
\label{gen_sf}
\end{gather}
for some function $\varphi: [0,\infty) \rightarrow [0,\infty)$ of one variable. Since the survival function of $\mu$ is invariant with respect to changes in the $\ell_p$-norm of its argument $\bm{x}$, it is called {$\ell_p$-norm symmetric}. Probability measures $\mu$ of the kind (\ref{gen_sf}), as studied in \cite{charpentier14}, widely appear in many areas of applications including finance, risk management, and environmental sciences; we refer to \cite{chatelain20} for background, examples, and statistical inference. Non-exchangeable extensions of (\ref{gen_sf}) are discussed in \cite{hofert18,hofert20}.
\par
When integrating with respect to a measure in $\IR^d$, it is sometimes convenient to perform this integration in two steps: first integrate with respect to a ``direction'' specified by some measure on the (bounded) unit ball of some norm, second integrate with respect to a one-dimensional ``radial'' measure. The representation (\ref{gen_sf}) for the survival function of $\mu$ suggests that such a decomposition is possible with a ``directional'' measure on the unit ball of the $\ell_p$-norm and $\varphi$ accounting for the ``radial'' part. If $\mu$ is a probability measure, intuitively this means that in order to simulate a random vector $\bm{Z} \sim \mu$ one may first simulate a bounded random vector taking values within the unit ball of the $\ell_p$-norm, and subsequently multiply this random vector with an independent radius. We also consider the case of non-finite Radon measures $\mu$ on $[\bm{0},\bm{\infty}] \setminus \{\bm{0}\}$, meaning that $\mu$ is non-finite but assigns finite measure to all closed sets that are bounded away from the origin $\bm{0}$. These play an important role in the context of {max-infinitely divisible probability distributions}, which are parameterized in terms of such a Radon measure, called the {exponent measure}, see \cite{resnick87} for a textbook account. A random vector is max-infinitely divisible if for arbitrary $n \in \IN$ it can be represented in distribution as the component-wise maximum of $n$ independent and identically distributed random vectors. Canonical stochastic representations for such random vectors rely on the notion of Poisson random measure with mean measure $\mu$. If one accomplishes a decomposition into directional and radial parts for the mean measure $\mu$, this can be leveraged to construct an exact simulation algorithm of the associated max-infinitely divisible probability law, as we will demonstrate for $\mu$ satisfying (\ref{gen_sf}). 
\par
In order to prepare the reader for the technical tools involved in the present article, it is instructive to notice that the function $\varphi$ in (\ref{gen_sf}) is necessarily $d$-monotone. We recall that a function $\varphi:(0,\infty) \rightarrow [0,\infty)$ is {$d$-monotone} if the derivatives $\varphi^{(k)}$ exist for $k\in \{0,\ldots,d-2\}$, and $(-1)^k\,\varphi^{(k)}$ is non-negative, non-increasing and convex. Such functions play an important role not only in the context of $\ell_p$-norm symmetric multivariate survival functions, but also in the context of $\ell_p$-symmetric multivariate characteristic functions, see \cite{gneiting98}. A result of R.E.\ Williamson in \cite{williamson56} provides a representation for $d$-monotone functions as integrals over certain simple functions with respect to a uniquely associated probability distribution on $[0,\infty)$. Thus, they arise as analytical transforms of probability measures, generalizing the notion of a Laplace transform in a certain sense. In general, this transform can be inverted to obtain the associated probability distribution, but in concrete cases this inversion is not simple to figure out in a feasible form. Key to our results is the inversion of the $d$-monotone function $x \mapsto (1-x^{\theta})_+^{d-1}$, whose associated probability distribution is shown to be related to a finite mixture of certain beta distributions. 
\par
If $\mu$ with survival function (\ref{gen_sf}) is a probability measure on $(0,\infty)^d$, it follows from results in \cite{mcneil09,charpentier14} (this logic being recalled in (\ref{gen_ellpsurvfunc}) below) that $\varphi$ is $d$-monotone with $\varphi(0)=1$ and $\mu$ is the distribution of
\begin{gather*}
\bm{Z} \sim R\,V_{p}\,\big(\bm{U}^{(1)}\big)^{\theta}, 
\end{gather*}
where $R$ is a random variable on $(0,\infty)$ whose distribution $F_{\varphi}$ depends solely on $\varphi$, the vector $\bm{U}^{(1)}$ is uniformly distributed over the standard unit simplex
 $S_{d,1}:=\{\bm{x} \in [0,1]^d\,:\,\norm{\bm{x}}_1=1\}$,
 $V_p$ is a random variable on $[0,1]$ whose distribution $F_p$ solely depends on $p$,
 and $R$, $V_p$ and $\bm{U}^{(1)}$ are independent. Here and throughout, raising a vector to a power $\theta$, as well as applying other functions of one variable to a vector, should always be understood component-wise. Equivalently, $\mu$ can be factored as 
\begin{gather}
\mu(A) = \int_{(0,\infty)}\int_{[0,1]}\,\int_{S_{d,1}} 1_{\{r\,v\,\bm{u}^{\theta} \in A \}}\,\mathrm{d}\bm{u}\,\mathrm{d}F_p(v)\,\mathrm{d}F_{\varphi}(r),\quad A \subset (0,\infty)^d \mbox{ a Borel set}.
\label{factor_ident}
\end{gather}
The distribution function $F_p$ of $V_p$ has not been explicitly found to date. We derive an exact representation for this probability distribution, finding that a random variable $V_p \sim F_p$ satisfies the distributional equality
\begin{gather*}
V_p \sim  W_{(d+1-D_{\theta})},\quad \IP(D_{\theta}=i)=a_i^{(d)},\quad i\in \{1,\ldots,d\},
\end{gather*}
where, independently of $D_{\theta}$, we denote by  $W_{(1)} \leq \ldots \leq W_{(d-1)}$ the order statistics of independent standard uniform random variables $W_1,\ldots,W_{d-1}$   and $W_{(d)}=1$, and the mixture probabilities $(a_1^{(d)},\ldots,a_d^{(d)}) \in S_{d,1}$ can conveniently be computed by the recursive relationship
\begin{gather}
a^{(k)}_i  = a^{(k-1)}_{i}\,\theta\,\frac{k-i}{k-1}+a^{(k-1)}_{i-1}\,\Big( 1-\theta\,\frac{k-i+1}{k-1}\Big) ,\quad i \in \{1,\ldots,k\},\quad k\in \{2,\ldots,d\},
\label{eq:recur-a}
\end{gather}
with initial value $a_1^{(1)}=1$ and auxiliary notation $a_0^{(k-1)}=a_{k}^{(k-1)}=0$. This finding implies an efficient simulation algorithm for random vectors $\bm{Z} \sim \mu$. Existing simulation algorithms to date are either restricted by the assumption that $\varphi$ is completely monotone (which is a special case of $d$-monotone in which $\varphi$ is a Laplace transform), or rely on computations of partial derivatives as in \cite[Proposition 5.3]{charpentier14}, which is infeasible for large $d$.
\par
The case $p=1$, in which $V_p \equiv 1$, is well-established, see \cite{mcneil09}. It is convenient to study the dependence structure between components of $\bm{Z}$ in terms of the survival copula of $\bm{Z}$ in that case. The latter equals the distribution function of the random vector $\varphi(\bm{Z})$ and which is called an {Archimedean copula} with {Archimedean generator $\varphi$}, see \cite{mcneil09} for background. The particular choice $\varphi(x)=(1-x)_+^{d-1}$ corresponds to the random vector $\bm{Z}=\bm{U^{(1)}}$, i.e.\ $R \equiv 1$, and this choice minimizes the association between components of $\bm{Z}$. In other words, randomness of the radial variable $R$ increases the strength of association between components of $\bm{Z}$ when compared to a non-random radius. In the general case $p \geq 1$ the situation is analogous and $\bm{Z}=V_p\,(\bm{U}^{(1)})^{\theta}$ is a stochastic representation for the random vector whose association between its components is minimal among all random vectors with $\ell_p$-norm symmetric survival functions. Our exact representation for the law of $V_p$ thus implies a stochastic model for the $\ell_p$-norm symmetric survival function with minimal association between components. Interestingly, the survival copula of $\bm{Z}$ in the general case $p \geq 1$ is also an Archimedean copula, but with Archimedean generator given by $x \mapsto \varphi(x^{\theta})$.
\par
A further application of our findings concerns the case when $\mu$ in (\ref{gen_sf}) is a non-finite Radon measure on $[\bm{0},\bm{\infty}] \setminus\{\bm{0}\}$. The results in \cite{genest18} imply that the function $\varphi$ is $d$-monotone and a bijection on $(0,\infty)$. The same factorization (\ref{factor_ident}) of $\mu$ is valid, only with the probability distribution $F_{\varphi}$ being replaced by a non-finite ``radial'' Radon measure $\nu_{\varphi}$ on $(0,\infty]$ that solely depends on $\varphi$. As already mentioned, our explicit derivation of the probability distribution $F_p$ can be leveraged to derive an exact simulation algorithm for max-infinitely divisible random vectors with exponent measure $\mu$. To the best of our knowledge, such algorithm is unknown to date.
\par
The remainder of this article is organized as follows. Section \ref{sec_prelim} provides background on different concepts of $\ell_p$-norm symmetry in the context of multivariate probability distributions. Section \ref{sec_main} derives the explicit form of the aforementioned distribution $F_p$. Section \ref{sec_sf} presents an efficient simulation algorithm for $F_p$ and thus for arbitrary random vectors with $\ell_p$-norm symmetric survival functions. Section \ref{sec_maxid} proves a stochastic representation for max-infinitely divisible random vectors $\bm{Y}$ on $(0,\infty)^d$ whose exponent measure has $\ell_p$-norm symmetric survival function, and explains how to simulate $\bm{Y}$ exactly from it. 

\section{Background on $\ell_p$-norm symmetry} \label{sec_prelim}
Concerning analytical characterizations of multivariate probability distributions, one can find three prominent notions of $\ell_p$-norm symmetry in the literature: $\ell_p$-norm symmetric densities, $\ell_p$-norm symmetric characteristic functions, and $\ell_p$-norm symmetric survival functions. In order to classify the contributions of the present article, we provide a concise survey of what is known about stochastic representations related to these concepts.
\par
First of all, an absolutely continuous random vector $\bm{X}$ has an $\ell_p$-norm symmetric density $f(\bm{x})=g(\norm{\bm{x}}_p)$ if and only if   
\begin{gather}
\bm{X} \sim R\,\bm{U}^{(p)},
\label{stoch_ellsymchar}
\end{gather}
where $R$ is a positive (absolutely continuous) random variable and $\bm{U}^{(p)}$ is an independent random vector that is uniformly distributed on the $\ell_p$-sphere, see Lemma \ref{lemma_lpcharsym} in the Appendix. Since $\norm{.}_p$ is orthant-monotonic, this statement holds either for $\bm{X}$ taking values in $\IR^d$ or only $(0,\infty)^d$, in which case $\bm{U}^{(p)}$ is uniform on the restriction of the $\ell_p$-sphere to the positive orthant (the $\ell_p$-simplex). A stochastic representation for $\bm{U}^{(p)}$ restricted to the positive orthant can be found in \cite{rachev91} and is given by $\bm{U}^{(p)} \sim \bm{\xi}^{(p)}/\norm{\bm{\xi}^{(p)}}_p$, where $\bm{\xi}^{(p)}=(\xi^{(p)}_1,\ldots,\xi^{(p)}_d)^T$ is a vector with iid components satisfying $(\xi_1^{(p)})^{p} \sim \Gamma(\theta,\theta)$, where we denote by $\Gamma(\beta,\eta)$ the Gamma distribution with density proportional to $e^{-\eta\,x}\,x^{\beta-1}$. 
\par
When considering analytical characterizations other than multivariate densities, the stochastic model (\ref{stoch_ellsymchar}) is replaced by a more complicated stochastic representation, except for the following two well-known special cases $p \in \{1,2\}$:
\begin{itemize}
\item[(i)] The characteristic function of $\bm{X}$ is $\ell_2$-norm symmetric if and only if $\bm{X} \sim R\,\bm{U}^{(2)}$ for arbitrary $R>0$ independent of the random vector $\bm{U}^{(2)}$.
\item[(ii)] The survival function of $\bm{X}$ is $\ell_1$-norm symmetric if and only if $\bm{X} \sim R\,\bm{U}^{(1)}$ for arbitrary $R>0$ independent of the random vector $\bm{U}^{(1)}$.
\end{itemize}
When generalizing these statements to the case of more general $p \neq 1$, the stochastic representations become more involved than (\ref{stoch_ellsymchar}), or are even unknown. In the present article we are concerned with a generalization of statement (ii) to the general case $p \geq 1$, whereas generalizations of statement (i) are beyond the scope of the present work. For background on (i) we refer the interested readers to the references in Remark \ref{rmk_char} below, which provides a short survey on the topic.
\par
Concerning the generalization of statement (ii), it follows from results in \cite{charpentier14,mcneil09}, this logic being explained in (\ref{gen_ellpsurvfunc}) below, that a random vector $\bm{Z}$ taking values in $(0,\infty)^{d}$ has $\ell_p$-norm symmetric survival function $\IP(\bm{Z}>\bm{z})=\varphi \circ \norm{\bm{z}}_p$ if and only if 
\begin{gather}
\bm{Z} \sim R\,V_{p}\,\big(\bm{U}^{(1)}\big)^{\theta},
\label{stoch_ellsymsf}
\end{gather}
where $R$ is a positive random variable uniquely determined in law by its so-called {Williamson-$d$-transform} (see \cite{williamson56}) 
\begin{gather*}
\IE\Big[\Big( 1-\frac{x}{R}\Big)^{d-1}_+ \Big] = \varphi(x),\quad x \geq 0,
\end{gather*}
$\bm{U}^{(1)}$ is uniform on $S_{d,1}$, and $V_{p}$ is a random variable taking values in $[0,1]$ whose probability law is not explicitly known to date, all three objects mutually independent. Provided absolute continuity of the radial variable $R$, the two representations (\ref{stoch_ellsymchar}) and (\ref{stoch_ellsymsf}) imply that the notions of $\ell_p$-norm symmetric densities and $\ell_p$-norm symmetric survival functions are equivalent if and only if $p=1$, and $V_1 \equiv 1$ in that case. The survival copula of $\bm{X}$, respectively $\bm{Z}$, in this case $p=1$, which equals the distribution function of the random vector $\varphi(\bm{X})$, is an Archimedean copula, and is given by
\begin{gather*}
C_{\varphi}(\bm{u}) = \varphi[ \varphi^{-1}(u_1)+\ldots+\varphi^{-1}(u_d)],\quad \bm{u} \in [\bm{0},\bm{1}],
\end{gather*}
and the function $\varphi$ is called {Archimedean generator}, see \cite[Chapter 2]{mai17} for background on the matter. In the general case $p \geq 1$, we see that $\bm{Z}^{p} \sim R^{p}\,V_p^{p}\,\bm{U}^{(1)}$ is a particular instance of an $\ell_1$-norm symmetric distribution (in both considered meanings). The survival copula of the vector $\bm{Z}$ in (\ref{stoch_ellsymsf}) is an Archimedean copula with Archimedean generator of the form $x \mapsto \varphi(x^{\theta})$, and these are sometimes referred to as {outer power Archimedean copulas}. This parameter-enhancement technique, introducing the power $\theta$ inside the argument of $\varphi$, has originally been introduced in \cite{oakes94}. The nomenclature ``outer power'' might appear surprising, since the power $\theta$ is taken ``inside'' $\varphi$, but is explained from traditional notation in the context of Archimedean copulas, where the roles of $\varphi$ and $\varphi^{-1}$ are often interchanged.
\par
Now let us briefly discuss the strength of association between components of $\bm{X}$ in (\ref{stoch_ellsymchar}), respectively $\bm{Z}$ in (\ref{stoch_ellsymsf}). Considering $\bm{X}$ in (\ref{stoch_ellsymchar}), for non-random $R$ the association between the components of $\bm{X}$ is minimal and negative; indeed, in this case, $\bm X^p$ is a joint mix \cite{WW16} and hence it represents a form of extreme negative dependence \cite{PW15}. Furthermore, it is explained in \cite{rachev91} that if the radial variable satisfies $R \sim M\,\norm{\bm{\xi}^{(p)}}_p$ with a positive random variable $M$ independent of $\bm{\xi}^{(p)}$, then $\bm{X} \sim M\,\bm{\xi}^{(p)}$ with $M$ and $\bm{\xi}^{(p)}$ independent. Intuitively, the denominator in $\bm{U}^{(p)} \sim \bm{\xi}^{(p)}/\norm{\bm{\xi}^{(p)}}_p$ ``cancels out in distribution'' in this case, relying on the Lukacs theorem, and the components exhibit positive association (recall that $\xi^{(p)}_1,\dots,\xi^{(p)}_d$ are iid) whose strength depends on $M$. Regarding the strength of dependence between components of $\bm{Z}$ in (\ref{stoch_ellsymsf}), a convenient measurement is Kendall's tau between a pair of two components of $\bm{Z}$, see \cite[p.\ 28--30]{mai17} for a motivation. We recall that {Kendall's tau} for the bivariate random vector $(Z_1,Z_2)$ is given by the probability of concordance minus the probability of discordance, that is
\begin{gather*}
\IP[(Z_1-\tilde{Z_1})\,(Z_2-\tilde{Z}_2)>0]-\IP[(Z_1-\tilde{Z_1})\,(Z_2-\tilde{Z}_2)<0] ,
\end{gather*}
where $(\tilde{Z}_1,\tilde{Z}_2)$ is an independent copy of $(Z_1,Z_2)$. It follows from the results in \cite{charpentier14} that Kendall's {tau} between two components of $\bm{Z}$, without loss of generality $(Z_1,Z_2)$ by exchangeability of $\bm{Z}$, is given by $1-\theta+\theta\,\tau_{\varphi}$, where $\tau_{\varphi}$ denotes Kendall's {tau} between two components of a random vector with Archimedean copula $C_{\varphi}$ as distribution function. 
It is known from results in \cite{mcneil09} that Kendall's {tau} $\tau_{\varphi}$ is minimized with the choice $\varphi(x)=(1-x)^{d-1}_+$ corresponding to $R \equiv 1$, with $\tau_{\varphi}=-1/(2\,d-3)$. For $p<(2\,d-2)/(2\,d-3)$ this implies negative association between the components of $\bm{Z}$ and we obtain a similar intuition as in the case of an $\ell_p$-norm symmetric density. In particular, two components of the random vector $\bm{Z}=V_p\,(\bm{U}^{(1)})^{\theta}$ have minimal Kendall's tau $1-\theta\,(2\,d-2)/(2\,d-3)$ among all $d$-dimensional random vectors with $\ell_p$-norm symmetric survival function. Furthermore, it is known that if $R \sim M\,\big(\norm{\bm{\xi}^{(1)}}_1\big)^{\theta}$ for some positive random variable $M$ independent of $\bm{\xi}^{(1)}$ as defined above, then $\bm{Z} \sim M\,\big(\bm{\xi}^{(1)}\big)^{\theta}$ and the components of $\bm{Z}$ exhibit positive association whose strength is governed by the choice of $M$. 
\par
Our main contribution is an explicit representation for the random variable $V_p$. It can be inferred from the results in \cite{mcneil09} that the random variable $V_{p}$ is uniquely determined by the identity
\begin{gather}
\IE\Big[\Big( 1-\frac{x}{V^p_p}\Big)^{d-1}_+ \Big] = \big( 1-x^{\theta}\big)^{d-1}_+,\quad x \geq 0.
\label{willi_V}
\end{gather}
Unfortunately, this Williamson-$d$-transform is not easy to invert to obtain the explicit law of $V^p_p$, hence $V_p$. We prove that $V_p$ equals a finite mixture of certain beta distributions and an atom at unity and derive an efficient simulation algorithm. This not only implies an efficient simulation algorithm for the random vector $\bm{Z}$ in (\ref{stoch_ellsymsf}), but also we show in Section \ref{sec_maxid} how it can be leveraged to obtain an exact simulation algorithm for max-infinitely divisible random vectors $\bm{Y}$ on $(0,\infty)^d$ whose exponent measure $\mu$ has $\ell_p$-norm symmetric survival function given by (\ref{gen_sf}). An excellent textbook account on max-infinite divisibility is \cite{resnick87}. Such $\bm{Y}$ is shown in Lemma \ref{lemma_rAC} below to have the stochastic representation
\begin{gather}
\bm{Y} \sim \Big( \max_{k \geq 1}\big\{ \eta_k\,Z_1^{(k)} \big\},\ldots, \max_{k \geq 1}\big\{ \eta_k\,Z_d^{(k)} \big\}\Big),
\label{stoch_maxid}
\end{gather}
where $\{\bm{Z}^{(k)}\}_{k \geq 1}$ is a sequence of iid copies of $\bm{Z}$ in (\ref{stoch_ellsymsf}) with $R \equiv 1$ and, independently, $\{\eta_k\}_{k \geq 1}$ denoting the decreasing enumeration of the points of a Poisson random measure, whose mean measure $\nu = \nu_{\varphi}$ is Radon on $(0,\infty]$ satisfying $\nu_{\varphi}(\{\infty\})=0$ and $\nu_{\varphi}((0,\infty])=\infty$. Since our main result implies an exact simulation algorithm for the involved $\bm{Z}^{(k)}$, this stochastic representation serves as basis to derive an exact simulation algorithm for $\bm{Y}$. Its idea enhances an algorithm viable for the case $p=1$ that was presented in \cite{mai18}. The copula of $\bm{Y}$ is called a {reciprocal Archimedean copula} with generator $x \mapsto \varphi(x^{\theta})$ in \cite{genest18}. This nomenclature is justified by some ``reciprocal'' analogies with Archimedean copulas, e.g., like Archimedean copulas also reciprocal Archimedean copulas can be written in terms of their generating function $\varphi$. Our algorithm shows how to simulate reciprocal Archimedan copulas whose generator is given by $y \mapsto \varphi(y^{\theta})$. In analogy to the aforementioned Archimedean case, we refer to the copula of $\bm{Y}$ as {outer power reciprocal Archimedean copula}. 

\begin{remark}[$\ell_p$-norm symmetric characteristic functions]\label{rmk_char}
Characteristic functions that are $\ell_2$-norm symmetric are popular in geostatistics. For instance, it is pointed out in \cite{gneiting99} that when the density of $R$ in (\ref{stoch_ellsymchar})  with $p=2$ is proportional to $J^2_{d/2}$, the square of a Bessel function, the resulting $\ell_2$-norm symmetric characteristic function (or density) is called {Euclid's hat}, and scale mixtures thereof constitute an important model in geostatistics. It is known that $\ell_p$-norm symmetric characteristic functions require the restriction $p \leq 2$, but an explicit stochastic representation for $p \neq 2$ is only known for $p=1$ due to \cite{cambanis83}, we refer the interested reader to \cite{gneiting98} for open questions in this regard and further background on the matter. 
\end{remark}

\section{Explicit representation for the law of $V_p$}\label{sec_main}

Concerning notation, we denote by $\beta_{m,n}$ for $m,n \geq 1$ the cdf of a beta distribution with density proportional to $x^{m-1}\,(1-x)^{n-1}$. For the sake of a convenient notation, we further denote by $\beta_{m,0}(x)=1_{\{x \geq 1\}}$ the cdf of a random variable that is identically constant equal to one, for $m \geq 1$ arbitrary.
\par
Our goal is to find the random variable $V_p$ satisfying (\ref{willi_V}). The solution will be given in Theorem \ref{thm_sol} below, where it is shown that $V_p$ is a (convex) mixture of certain beta distributions. Before presenting it, some auxiliary steps are carried out. First of all, for the sake of completeness, we formally prove that $V_p$ satisfying (\ref{willi_V}) exists and is unique in law. A result of \cite{williamson56}, lying at the heart of the results in \cite{mcneil09}, shows that functions $\varphi:[0,\infty) \rightarrow [0,1]$ which are $d$-monotone on $(0,\infty)$ and satisfy $\varphi(0)=1$ form a simplex with extremal boundary given by the functions $\varphi_v(x):=(1-x/v)^{d-1}_+$, $v>0$. In intuitive terms, this means that these functions form a compact convex set, and each element in this set has a unique representation as an ``integral average'' over functions in the boundary of the set. In probabilistic terms, it means that for any such function $\varphi$ there is a random variable $V_{\varphi}$, uniquely determined in distribution, such that $\varphi(x)=\IE[\varphi_{V_{\varphi}}(x)]$, $x \geq 0$. Applied to our situation, in order to formally prove that $V_p$ exists and its law is unique it is sufficient to verify that $x \mapsto (1-x^{\theta})^{d-1}_+$ is $d$-monotone.

\begin{lemma}[$x \mapsto (1-x^{\theta})^{d-1}_+$ is $d$-monotone]
The function $x \mapsto (1-x^{\theta})^{d-1}_+$ is $d$-monotone on $(0,\infty)$.
\end{lemma}
\begin{proof}
Denote $\varphi(x) = (1-x^{\theta})^{d-1}_+$ within this proof. We apply \cite[Theorem 12]{ressel14}, which states that $f \circ g$ is $d$-$\uparrow$ if both $f$ and $g$ are. Applying this statement with $f(x)=(1+x)^{d-1}_+$ and $g(x)=-(-x)^{\theta}$ on $(-1,0)$, which are both easily seen to be $d$-$\uparrow$, then implies that $\varphi=f \circ g(-.)$ is $d$-monotone on $(0,1)$. Since $\varphi^{(k)}$ is identically zero on $[1,\infty)$ and $\varphi^{(k)}(1)=0$ for $k\in \{0,\ldots,d-2\}$, we obtain that $\varphi^{(k)}$ is actually convex on all of $(0,\infty)$, hence $\varphi$ is $d$-monotone on $(0,\infty)$.
\end{proof}
By definition of $d$-monotonicity we also know that $x \mapsto (1-x^{\theta})^{d-1}_+$ is $k$-monotone on $(0,\infty)$ for each $k=1,\ldots,d$. Consequently, for each $k=1,\ldots,d$ there exists a positive random variable $V_d^{(k)}$, which is unique in law, such that
\begin{gather*}
\big(1-x^{\theta}\big)^{d-1}_+ = \IE\Big[ \Big(1-\frac{x}{\big(V_d^{(k)}\big)^p} \Big)^{k-1}_+\Big],\quad x >0.
\end{gather*}
Our goal is to determine the probability law of $V_p=V_d^{(d)}$, in fact we even determine the law of all $V_d^{(k)}$ for $k=1,\ldots,d$ in the following. We denote the cdf of $V_d^{(k)}$ by $F_d^k$ and, as a first step, we derive a recursion for $F_d^k$. To this end, we note that for $k=1,\dots,d$, $F_d^k$ is the unique distribution which satisfies the equation 
 \begin{align} 
\label{eq:recur1}
  \int^{1}_{c^{\theta}} \left(1- \frac{c}{x^p}  \right)^{k-1} \d F^{k}_{d}(x)    =  (1-c^{\theta})^{d-1},~~~\mbox{for $c\in [0,1]$.}
 \end{align}
 
\begin{lemma}[A recursion for $F_d^k$]\label{lemma_rec}
Let the finite variation functions {$F_d^k$} be given by the following recursive formulas: $F_d^1 = \beta_{1,d-1}$, and for $k \in\{ 2,3,\dots,d\}$,
 \begin{align} 
\label{eq:recur2}
F^{k}_d = \frac{d-1}{k-1} \, \theta\,  F^{k-1}_{d-1} + \left(1-\frac{d-1 }{k-1}\,\theta \right)    F^{k-1}_{d}. \end{align} 
Then, $F_d^k$ satisfies \eqref{eq:recur1}.
\end{lemma}
\begin{proof} 
We know $\beta_{1,d-1}(x)=1-(1-x)^{d-1}$, which implies that
\begin{gather*}
 \int^{1}_{c^{\theta}} \left(1- \frac{c}{x^p}  \right)^{0} \d \beta_{1,d-1}(x) =  \int^{1}_{c^{\theta}}\, \d \beta_{1,d-1}(x) = (1-c^{\theta})^{d-1},
\end{gather*}
as claimed for $k=1$. Regarding the induction step, for $k \geq 2$ define two functions $F(c) :=  \int^{1}_0   \left(1- \frac{c}{x^p}  \right)_+^{k-1} \d F^{k}_{d}(x) $ and $G(c) := (1-c^{\theta} )^{d-1}$ for $c\in [0,1]$. 
  Note that 
  $$
  F^{'}(c) =   \frac{\d }{\d c}\int_{c^{\theta}}^{1}   \left(1- \frac{c}{x^p}  \right)^{k-1} \d F^{k}_{d}(x)= -(k-1) \int_{c^{\theta}}^{1}  \frac{1}{x^p}   \left(1- \frac{c}{x^p}  \right)^{k-2} \d F^k_d(x).
  $$
Moreover, using \eqref{eq:recur2}, 
\begin{align*}
c\,F^{'}(c) &= -(k-1) \int_{c^{\theta}}^1  \frac{c}{x^p}   \left(1- \frac{c}{x^p}  \right)^{k-2} \d F^k_d(x) =-(k-1)  \left[\int_{c^{\theta}}^1    \left(1- \frac{c}{x^p}   \right)_+^{k-2} \d F^k_d(x)  -  F(c) \right] \\
&=   -\int_{c^{\theta}}^1    \left(1-\frac{c}{x^p}   \right)^{k-2} \left( (d-1)\,\theta \d F^{k-1}_{d-1}(x)   + [k-1- (d-1)\,\theta]\, \d F^{k-1}_d(x)\right)  +  (k-1) F(c) \\
&= -(d-1)\,\theta\,(1-c^{\theta} )^{d-2}   -  [k-1- (d-1)\theta]\, (1-c^{\theta} )^{d-1}+  (k-1) F(c) \\ 
&=-(d-1)\,  \theta\,  (1-c^{\theta} )^{d-2}\, c^{\theta} + (k-1)\,[ F(c)-G(c) ].
\end{align*} 
 We know $F(0)=G(0)$. If $c>0$, we divide both sides of the above equality by $c$, and  get
 $$
 F'(c) =G'(c) + \frac{k-1}{c}\,[F(c)-G(c) ].
 $$
 Since $F(1)=G(1)=0$, we know that $F \equiv G$ on $[0,1]$. Thus, \eqref{eq:recur1} holds. 
\end{proof}

The term $ \theta\,(d-1)/(k-1)$  in  \eqref{eq:recur2} may be greater than one, so that we do not obtain convex combinations of beta distributions directly. Indeed, if this term is no larger than one (i.e., $p\ge d-1$), then applying \eqref{eq:recur2} repeatedly gives rise  to $F_d^d$ as a mixture of $\beta_{1,k}$ for $k\in \{1,\dots,d-1\}$. In general, this is not the case: we will see that $F_d^d$ is a mixture of beta distributions, but not all of the form  $\beta_{1,k}$ for $k \in \{1,\dots,d-1\}$. The following auxiliary lemma is helpful to solve the recursion in \eqref{eq:recur2}.

\begin{lemma}[Auxiliary identities on the beta distribution]\label{lemma_ident}
The following two identities hold for the beta distribution, for integers $m,n \geq 1$:
\begin{align*}
\beta_{m+1,n-1}(x) - \beta_{m,n}(x) = -\binom{m+n-1}{m}\,x^m\,(1-x)^{n-1},\quad \beta_{m,n-1}(x)-\beta_{m,n}(x) =-\binom{m+n-2}{m-1}\,x^m\,(1-x)^{n-1}.
\end{align*}
\end{lemma}
\begin{proof}
The straightforward proof is sketched in the Appendix.
\end{proof}

\begin{theorem}[Solving the recursion]\label{thm_sol}
For each $k \in \{1,\ldots,d\}$, there exists $(a^{(k)}_1,\ldots,a^{(k)}_k) \in S_{k,1}$ such that
\begin{gather*}
F_d^{k} = \sum_{i=1}^{k}a^{(k)}_i\,\beta_{k+1-i,d-k-1+i}.
\end{gather*}
Furthermore, the $a_i^{(k)}$ satisfy the recursive relationship (\ref{eq:recur-a}).
\end{theorem}
\begin{proof}
If we fix $i \in \{1,\ldots,k-1\}$, then the second identity in Lemma \ref{lemma_ident} gives
\begin{align*}
 \frac{d-1}{k-1}\big(\beta_{k-i,d-k+i-1}(x)-\beta_{k-i,d-k+i}(x)\big) = -\frac{d-1}{k-1}\,\binom{d-2}{k-i-1}\,x^{k-i}\,(1-x)^{d-k+i-1}   = -\binom{d-1}{k-i}\,x^{k-i}\,(1-x)^{d-k+i-1}\,\underbrace{\frac{k-i}{k-1}}_{ \leq 1}.
\end{align*}
Consequently, we observe with the help of the first identity in Lemma \ref{lemma_ident}  that
\begin{align}
\frac{d-1}{k-1}\big(\beta_{k-i,d-k+i-1}-\beta_{k-i,d-k+i}\big)+\beta_{k-i,d-k+i}  = \frac{k-i}{k-1}\,\beta_{k-i+1,d-k+i-1}+\Big(1-\frac{k-i}{k-1}\Big)\,\beta_{k-i,d-k+i}. \label{decisive_step}
\end{align}
Now, inductively, we proceed as follows to compute $F_d^k$ via the recursion of Lemma \ref{lemma_rec}:
\begin{align*}
F_d^k  = \frac{d-1}{k-1}\,\theta\,F_{d-1}^{k-1}+\Big(1-\frac{d-1}{k-1}\,\theta \Big)F_d^{k-1}= \theta\,\Big[\frac{d-1}{k-1}\,\big(F^{k-1}_{d-1}-F^{k-1}_d \big)+F_d^{k-1}\Big]+(1-\theta)\,F_d^{k-1}.
\end{align*}
We know by induction that there exist $a^{(k-1)}_1 \geq 0,\ldots,a^{(k-1)}_{k-1} \geq 0$ that sum up to one and
\begin{gather*}
F_d^{k-1} = \sum_{i=1}^{k-1}a^{(k-1)}_i\,\beta_{k-i,d-k+i},\quad F_{d-1}^{k-1} = \sum_{i=1}^{k-1}a^{(k-1)}_i\,\beta_{k-i,d-1-k+i}.
\end{gather*}
Notice that we have used here that the $a_i^{(k-1)}$ are independent of $d$, which is important. We thus obtain 
\begin{align*}
F_d^k &= \theta\,\sum_{i=1}^{k-1}a^{(k-1)}_i\,\Big[\frac{d-1}{k-1}\,\big( \beta_{k-i,d-1-k+i}-\beta_{k-i,d-k+i}\big)+\beta_{k-i,d-k+i}\Big] +(1-\theta)\,F_d^{k-1}\\
& \stackrel{(\ref{decisive_step})}{=} \theta\,\sum_{i=1}^{k-1}a^{(k-1)}_i\,\Big[\frac{k-i}{k-1}\,\beta_{k-i+1,d-k+i-1}+\Big(1-\frac{k-i}{k-1}\Big)\,\beta_{k-i,d-k+i}\Big]+(1-\theta)\,F_d^{k-1}\\
& = \sum_{i=1}^{k-1}a^{(k-1)}_i\,\Big\{\theta\,\Big[\frac{k-i}{k-1}\,\beta_{k-i+1,d-k+i-1}+\Big(1-\frac{k-i}{k-1}\Big)\,\beta_{k-i,d-k+i}\Big] +(1-\theta)\,\beta_{k-i,d-k+i}\Big\}\\
& = \sum_{i=1}^{k-1}a^{(k-1)}_i\,\theta\,\frac{k-i}{k-1}\,\beta_{k-i+1,d-k+i-1}+\sum_{i=1}^{k-1}a^{(k-1)}_i\,\Big( 1-\theta\,\frac{k-i}{k-1}\Big)\,\beta_{k-i,d-k+i}\\
& = \sum_{i=0}^{k-2}a^{(k-1)}_{i+1}\,\theta\,\frac{k-i-1}{k-1}\,\beta_{k-i,d-k+i}+\sum_{i=1}^{k-1}a^{(k-1)}_i\,\Big( 1-\theta\,\frac{k-i}{k-1}\Big)\,\beta_{k-i,d-k+i}\\
& =\sum_{i=2}^{k-1}\Big[a^{(k-1)}_{i}\,\theta\,\frac{k-i}{k-1}+a^{(k-1)}_{i-1}\,\Big( 1-\theta\,\frac{k-i+1}{k-1}\Big) \Big]\,\beta_{k-i+1,d-k+i-1} + a^{(k-1)}_1\,\theta\,\beta_{k,d-k}+a^{(k-1)}_{k-1}\,\Big(1-\theta\,\frac{1}{k-1} \Big)\,\beta_{1,d-1}.
\end{align*}
This implies the claim.
\end{proof}

Apparently, $a_1^{(d)}=p^{-(d-1)}$. Since $\beta_{d,0} = \delta_1$ by our convenient notation, this implies that $V_p = V_d^{(d)}$ is equal to one with probability $a_1^{(d)}$ and with complementary probability $1-a_1^{(d)}$ follows an absolutely continuous distribution with support $[0,1]$. 

\section{Simulation of $\ell_p$-norm symmetric survival functions}\label{sec_sf}

Based on Theorem \ref{thm_sol}, we first derive a convenient method to simulate $V_p$ exactly. 
  
 \begin{lemma}[Simulating $V_p$] 
Let $W_1,\dots,W_{d-1}$ be iid from $\mathcal U[0,1]$, and  $W_{(i)}$ be the $i$-th order statistics, from the smallest to the largest, and $W_{(d)}:=1$. Define a counting process $(N_k)_{k=1}^d$ independent of $(W_1,\dots,W_{d-1})$ via $N_k=\sum_{j=1}^k B_j$, where $B_1=1$ and  for $j\in \{ 2,\dots,d\}$, 
 $$\IP(B_j=1\,|\,N_{j-1})  = 1-\IP(B_j=0\,|\,N_{j-1}) =  \frac{ N_{j-1}  }{j-1}\,\theta.$$ 
 Then $  W_{(N_k)} \sim F_d^k$.  
\end{lemma}

 \begin{proof}
 Let $T_k=k+1-N_k=1+ \sum_{j=1}^k (1-B_j)$, $k \in \{1,\dots,d\}$.
  Note that $T_1=1$. For $k \in \{2,\dots,d\}$ and $i \in \{1,\dots,k\}$, 
 \begin{align*}
  \IP(T_k = i) & =  \IP(N_k = k-i+1)  =  \IP(N_{k-1} = k-i,~B_k=1)  +  \IP(N_{k-1} =k-i+1,~B_k=0)  \\
&=  \IP(N_{k-1} = k- i) \frac{k-i}{k-1}\theta  +  \IP(N_{k-1} =k-i+1)    \left(1-\frac{k-i+1}{k-1}\theta\right)
 \\&  =  \IP(T_{k-1} = i) \frac{k-i}{k-1}\theta  +  \IP(T_{k-1} =i-1)    \left(1-\frac{k-i+1}{k-1}\theta\right).
 \end{align*}
Hence, the sequence $(\IP(T_k=i): 1\le i\le k\le d)$  satisfies the recursive relation \eqref{eq:recur-a} and has the same initial element as $(a_{i}^{(k)}: 1\le i\le k\le d)$. As a consequence, $a_{i}^{(k)}= \IP(T_k=i)$ for each $i$ and $k$.  
 Note that $W_{(k+1-i)}\sim \beta_{k+1-i,d-k-1+i}$ for $i=1,\dots,k$.
The law of total probability implies     
$$  W_{(k+1-T_k)} \sim \sum_{i=1}^k \IP(T_k=i) \beta_{k+1-i,d-k-1+i} = \sum_{i=1}^k  a_{i}^{(k)} \beta_{k+1-i,d-k-1+i} = F_d^k. $$
Thus, $W_{(N_k)}\sim F_d^k$.
 \end{proof} 

Algorithm \ref{algo} summarizes our simulation algorithm for $V_p$. The sub-routine \textproc{SimulateU}$[0,1](n)$ denotes a simulation algorithm for a list of $n$ iid uniform variates on $[0,1]$.

\begin{algorithm}
\caption{Simulation of $V_p$}\label{algo}
\begin{algorithmic}[1]
\Procedure{SimulateVp}{$p,d$}
\State $\bm{W}=(W_1,\ldots,W_{d-1}) \gets$\textproc{SimulateU}$[0,1](d-1)$ 
\State $\bm{W} \gets$\textproc{Sort}$(\bm{W})$
\State $\bm{W} \gets (\bm{W},1)$
\State $N \gets 1$
\For{$j=2,\ldots,d$}
\State $B \gets 0$
\State $U \gets$\textproc{SimulateU}$[0,1](1)$
\If{$U<\theta\,\frac{N}{j-1}$}
\State $B \gets 1$
\EndIf
\State $N \gets N + B$
\EndFor
\State \Return $V_p \gets W_{N}$
\EndProcedure
\end{algorithmic}
\end{algorithm}

Now denote by $V_p$ a random variable satisfying $V^p_p \sim F_d^d$, for instance simulated via Algorithm \ref{algo}, and denote by $\bm{U}^{(1)}$ an independent random vector that is uniformly distributed on the standard unit simplex in $[0,1]^d$, for instance simulated using the stochastic representation $\bm{U}^{(1)} \sim \bm{\xi}^{(1)}/\norm{\bm{\xi}^{(1)}}_1$ relying on a simulation of $d$ iid unit exponentials. We consider the random vector $\bm{Z} = V_p\,\big(\bm{U}^{(1)}\big)^{\theta}$, and observe that by construction
\begin{align*}
\IP(\bm{Z}>\bm{z}) &= \IP\Big(\bm{U}^{(1)}>\frac{\bm{z}^p}{V_p^p}\Big) = \IE\Big[ \Big( 1-\frac{\norm{\bm{z}^p}_1}{V_p^p}\Big)^{d-1}_+\Big] = \IE\Big[ \Big( 1-\frac{\norm{\bm{z}}_p^{p}}{V_p^p}\Big)^{d-1}_+\Big]\\
& = \int_0^{1}\Big( 1-\frac{\norm{\bm{z}}_p^{p}}{v^p}\Big)_+^{d-1}\,\mathrm{d}F^d_d(v) \stackrel{(\ref{eq:recur1})}{=} \big(1-\norm{\bm{z}}_{p} \big)^{d-1}_+.
\end{align*}
More generally, let now $\varphi$ be an arbitrary, non-negative $d$-monotone function with $\varphi(0)=1$, and denote by $R_{\varphi}$ a random variable, unique in law, satisfying
\begin{gather*}
\IE\Big[ \Big(1-\frac{x}{R_{\varphi}}\Big)^{d-1}_+\Big] = \varphi(x),\quad x \geq 0,
\end{gather*}
independent of $V_p$ and $\bm{U}^{(1)}$. Then the random vector $\bm{Z}=R_{\varphi}\,V_p\,\big(\bm{U}^{(1)}\big)^{\theta}$
satisfies
\begin{gather}
\IP(\bm{Z}>\bm{z}) = \IP\Big( V_p\,\big(\bm{U}^{(1)}\big)^{\theta} > \frac{\bm{z}}{R_{\varphi}}\Big) =  \IE\Big[\Big(1-\frac{\norm{\bm{z}}_{p}}{R_{\varphi}} \Big)^{d-1}_+ \Big] = \varphi\big(\norm{\bm{z}}_{p} \big),
\label{gen_ellpsurvfunc}
\end{gather}
as desired. 
\begin{example}[Simulation of strict outer power Clayton copulas] \label{ex_AC}
Consider the Archimedean generator $\varphi(x)=(1-x/a)_+^{a}$ for a parameter $a \geq d-1$, which is known as a {strict Clayton generator}. In \cite[Example 3.3]{mcneil09} this is shown to be $d$-monotone and it is also shown that the distribution function of $R_{\varphi}$ is given by
\begin{gather*}
\IP(R_{\varphi} \leq x) = 1-\sum_{k=0}^{d-1}\frac{a\,(a-1)\,\cdots\,(a-k+1)}{k!}\,\Big( \frac{x}{a}\Big)^{k}\,\Big( 1-\frac{x}{a}\Big)^{a-k}, \quad x \in [0,a].
\end{gather*}
Taking the derivative, it is not difficult to compute from this expression that for $a>d-1$ the random variable $R_{\varphi}$ satisfies the distributional equality $R_{\varphi} /a \sim \beta_{d,a-d+1}$. Our results imply that the random vector $\bm{Z} \sim R_{\varphi} \,V_p\,\big(\bm{U}^{(1)}\big)^{\theta}$ has survival function $(1-\norm{.}_p/a)^{a}$. Consequently, the distribution function of the random vector $\big(\varphi(Z_1),\ldots,\varphi(Z_d)\big)$ is the Archimedean copula with $x \mapsto \varphi(x^{\theta})=(1-x^{\theta}/a)^{a}$ as Archimedean generator. This is a strict outer power Clayton copula. Fig.\ \ref{fig:AC} shows scatter plots for this copula in the case $d=2$ (because larger $d$ are difficult to visualize), which have been produced making use of Algorithm \ref{algo}.
\end{example}

\begin{figure}[h]
\centering
\includegraphics[width=0.4\linewidth]{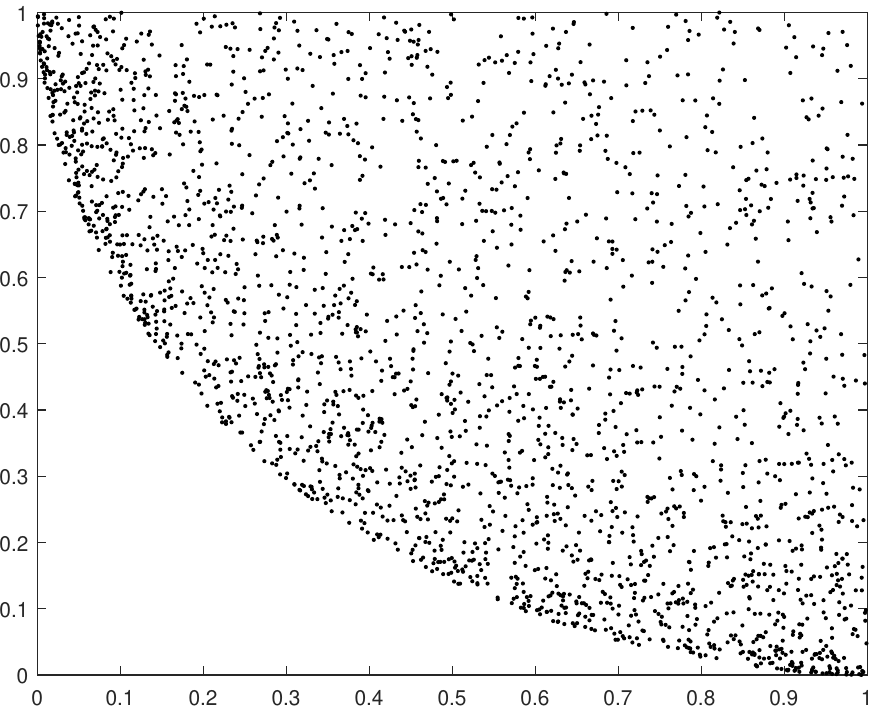}
\hfill
\includegraphics[width=0.4\linewidth]{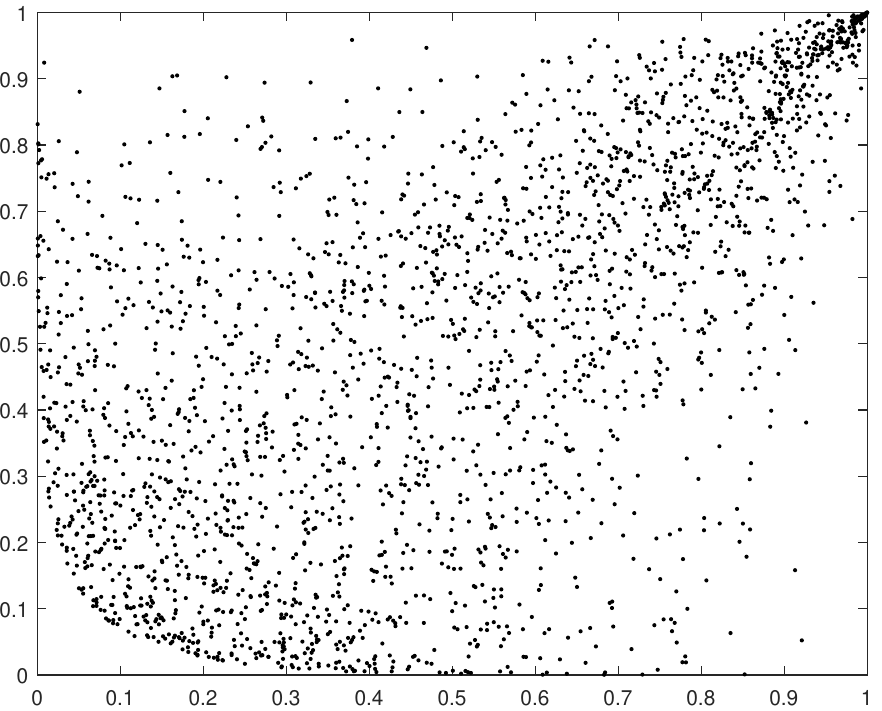}
\caption{Scatter plot of $2,500$ samples from the strict outer power Clayton copula in Example \ref{ex_AC}. Left: $p=1$ (so proper Clayton) and $a=1.75$. Right: $p=2.5$ and $a=1.75$.}
\label{fig:AC}
\end{figure}

\begin{remark}[Relation to positive stable distribution]
If $\varphi(x)=\exp(-x)$, it is well-known and easy to verify that $\bm{Z}$ with survival function $\exp(-\norm{.}_p)$ satisfies $\bm{Z} \sim M_{\theta}^{-\theta}\,(\bm{\xi}^{(1)}\big)^{\theta}$, where $M_{\theta}$ is a positive stable random variable with Laplace transform $x \mapsto \exp(-x^{\theta})$. Since $\varphi$ equals the Williamson-$d$-transform of an Erlang distributed random variable $E$ with $d$ degrees of freedom, our results thus imply the distributional identity
\begin{gather*}
M_{\theta}^{-\theta}\,\big(\bm{\xi}^{(1)}\big)^{\theta} \sim E\,V_p\,\Big(\frac{\bm{\xi}^{(1)}}{\norm{\bm{\xi}^{(1)}}_1}\Big)^{\theta}.
\end{gather*}
Since $V_p$ is a finite mixture of beta distributions, this resembles a distributional equality found in \cite[Theorem 1]{simon14}, representing the positive stable distribution with rational $\theta$ in terms of beta distributions.
\end{remark}

\section{Max-infinitely divisible laws with $\ell_p$-norm symmetric exponent measures} \label{sec_maxid}
A random vector $\bm{Y}$ taking values in $(0,\infty)^d$ is called {max-infinitely divisible} if for arbitrary $n \geq 1$ there exist iid random vectors $\bm{Y}^{(1,n)},\ldots,\bm{Y}^{(n,n)}$ such that
\begin{gather*}
\bm{Y} \sim \Big(\max_{i=1,\ldots,n}\big\{ Y^{(i,n)}_1\big\},\ldots,\max_{i=1,\ldots,n}\big\{ Y^{(i,n)}_d\big\} \Big).
\end{gather*}
To a large extent, a theory for max-infinitely divisible probability distributions can be obtained analogous to the theory for infinitely divisible probability distributions, when replacing the group operation ``addition'' in the latter with the semi-group operation ``maximum''. General stochastic representations in these theories rely on the notion of Poisson random measures and a textbook account on the topic is \cite{resnick87}. An analytical treatment for a Poisson random measure is given in terms of a Radon measure, and this carries over to a parameterization of the associated max-infinitely divisible probability distribution. Indeed, it is well known that $\bm{Y}$ is max-infinitely divisible if and only if its distribution function is given by
\begin{gather*}
\IP(\bm{Y} \leq \bm{y}) = \exp\Big[ - \mu\Big( E \setminus [\bm{0} ,\bm{y}] \Big)\Big],
\end{gather*}
where $\mu$ is a measure on $E:=[\bm{0},\bm{\infty}] \setminus\{\bm{0}\}$ subject to the properties
\begin{align*}
\mu\big( E \setminus [\bm{0},\bm{y}] \big)< \infty\,\quad \forall \bm{y}>\bm{0},\quad \lim_{\bm{y} \rightarrow \bm{\infty}}\mu\big(E \setminus  [\bm{0},\bm{y}] \big) = 0.
\end{align*}
The measure $\mu$ is called the {exponent measure} of $\bm{Y}$ and exponent measures $\mu$ with $\ell_1$-norm symmetric survival function are investigated in \cite{genest18}. We generalize this investigation to $\ell_p$-norm symmetric survival functions in the following. To wit, we say that $\mu$ has an $\ell_p$-norm symmetric survival function if there is a function $\varphi:(0,\infty) \rightarrow [0,\infty)$ in one variable, called {generator}, such that
\begin{gather*}
\mu\big( (\bm{y},\bm{\infty}]\big) = \varphi(\norm{\bm{y}}_p),\quad \bm{y} \in E.
\end{gather*}

As explained in \cite{genest18}, with Poincar\'e's inclusion exclusion identity we may write
\begin{align*}
\mu\Big( E \setminus [\bm{0} ,\bm{y}] \Big) = \sum_{\emptyset \neq I \subset \{1,\ldots,d\}}^{d}(-1)^{|I|+1}\,\mu\Big( ( \bm{y}_I,\bm{\infty}]\Big),
\end{align*}
where $\bm{y}_I \in [\bm{0},\bm{\infty})$ denotes a point whose $j$-th coordinate equals $y_j\,1_{\{j \in I\}}$. If now $\mu$ has $\ell_p$-norm symmetric survival function, then we obtain
\begin{align*}
\mu\Big( E \setminus [\bm{0} ,\bm{y}] \Big) = \sum_{\emptyset \neq I \subset \{1,\ldots,d\}}^{d}(-1)^{|I|+1}\,\varphi(\norm{\bm{y}_I}_{p}),
\end{align*}
so that the distribution function of $\bm{Y}$ is given in terms of the univariate function $\varphi$. Furthermore, it is immediately clear from this computation that $\bm{Y}$ is max-infinitely divisible with $\ell_p$-norm symmetric survival function and generator $\varphi$ if and only if the random vector $\bm{Y}^{p}$ is max-infinitely divisible with $\ell_1$-norm symmetric survival function and generator $x \mapsto \varphi(x^{\theta})$. The following lemma gives a concise recap of the results in \cite{genest18}.

\begin{lemma}[Genest et al., 2018 \cite{genest18}]\label{lemma_max_id}
Fix $p \geq 1$. The following are equivalent for a function $\varphi:(0,\infty) \rightarrow [0,\infty)$:
\begin{itemize}
\item[(a)] There exists a non-finite Radon measure $\nu$ on $(0,\infty]$ with $\nu(\{\infty\})=0$ such that 
\begin{gather*}
\varphi(t) =\varphi_{\nu}(t):= \int_t^{\infty}\Big( 1-\frac{t}{r}\Big)^{d-1}_+\,\nu(\mathrm{d}r).
\end{gather*}
\item[(b)] $\varphi$ is $d$-monotone and satisfies $\lim_{t \rightarrow \infty}\varphi(t) = 0$, $\lim_{t \searrow 0}\varphi(t) = \infty$.
\item[(c)] $\varphi$ is the generator of a max-infinitely divisible law on $(0,\infty)^d$ whose exponent measure has $\ell_1$-norm symmetric survival function.
\item[(d)] $\varphi$ is the generator of a max-infinitely divisible law on $(0,\infty)^d$ whose exponent measure has $\ell_p$-norm symmetric survival function.
\end{itemize}
\end{lemma}
\begin{proof}
The equivalences of (a) - (c) have been established in \cite{genest18}, and that (d) is equivalent as well has been explained in the text preceding this lemma.
\end{proof}

Making use of Theorem \ref{thm_sol}, we are able to derive an exact simulation algorithm for max-infinitely divisible $\bm{Y}$ whose exponent measure satisfies (\ref{gen_sf}). The basis for this algorithm is the following lemma.

\begin{lemma}[Stochastic representation for $\bm{Y}$]\label{lemma_rAC}
If any of the conditions in Lemma \ref{lemma_max_id} is satisfied, a max-infinitely divisible random vector $\bm{Y}$ whose exponent measure $\mu$ is given by (\ref{gen_sf}) satisfies the distributional equality
\begin{gather}
\bm{Y} \sim \Big( \max_{k \geq 1}\big\{ G_{\nu}^{-1}(\xi_1+\ldots+\xi_k)\,Z_1^{(k)} \big\},\ldots, \max_{k \geq 1}\big\{ G_{\nu}^{-1}(\xi_1+\ldots+\xi_k)\,Z_d^{(k)} \big\}\Big),
\label{shortref}
\end{gather}
where $G_{\nu}(x):=\nu\big( (x,\infty]\big)$ denotes the survival function of the Radon measure $\nu$ in Lemma \ref{lemma_max_id}(a) and $G_{\nu}^{-1}$ its generalized inverse, $\xi_1,\xi_2,\ldots$ is a sequence of iid standard exponential random variables and, independently, $\bm{Z}^{(1)},\bm{Z}^{(2)},\ldots$ is a sequence of iid copies of $\bm{Z} \sim V_p\,\big(\bm{U}^{(1)}\big)^{\theta}$.
\end{lemma}
\begin{proof}
Notice that $P := \sum_{k \geq 1}\delta_{(\xi_1+\ldots+\xi_k,\bm{Z}^{(k)})}$ is Poisson random measure on $[0,\infty) \times [\bm{0},\bm{1}]$ with mean measure $\mathrm{d}x\,\times \,\IP(\bm{Z} \leq \mathrm{d}\bm{z})$. We denote by $\tilde{\bm{Y}}$ the random vector on the right-hand side of (\ref{shortref}), and we compute with the exponential functional formula for the Poisson random measure, see \cite{resnick87}, and with the help of inclusion exclusion that 
\begin{align*}
& -\log\big[\IP(\tilde{\bm{Y}} \leq \bm{y})\big]  = -\log\Big(\IE\Big[ \exp\Big\{ -\int -\log\Big( \prod_{i=1}^{d}1_{\{G_{\nu}^{-1}(x)\,z_i \leq y_i\}}\Big)\,P(\mathrm{d}x,\mathrm{d}\bm{z}) \Big\}\Big]\Big)\\
& = \int_{(0,\infty)}\int_{[\bm{0},\bm{1}]} \Big(1- \prod_{i=1}^{d}1_{\{G_{\nu}^{-1}(x)\,z_i \leq y_i\}}\Big)\,\mathrm{d}x\,\IP(\bm{Z} \leq \mathrm{d}\bm{z}) = \IE\Big[ \max_{i=1,\ldots,d}\Big\{ G_{\nu}\Big( \frac{y_i}{Z_i}\Big)\Big\}\Big] \\
&\quad = \int_0^{\infty}\IP\Big(\max_{i=1,\ldots,d}\Big\{ G_{\nu}\Big( \frac{y_i}{Z_i}\Big)\Big\}>x\Big)\, \mathrm{d}x  = \int_0^{\infty}\IP\Big(\bigcup_{i=1}^{d}\Big\{ Z_i > \frac{y_i}{G_{\nu}^{-1}(x)}\Big\}\Big)\, \mathrm{d}x \\
& = \sum_{\emptyset \neq I \subset \{1,\ldots,d\}}(-1)^{|I|+1}\int_0^{\infty}\IP\Big( \bm{Z}> \frac{\bm{y}_I}{G_{\nu}^{-1}(x)}\Big)\, \mathrm{d}x  = \sum_{\emptyset \neq I \subset \{1,\ldots,d\}}(-1)^{|I|+1}\int_0^{\infty}\Big(1-\frac{\norm{\bm{y}_I}_p}{G_{\nu}^{-1}(x)}\Big)^{d-1}_+ \, \mathrm{d}x \\
&\quad = \sum_{\emptyset \neq I \subset \{1,\ldots,d\}}(-1)^{|I|+1}\int_0^{\infty}\Big(1-\frac{\norm{\bm{y}_I}_p}{x}\Big)^{d-1}_+ \, \nu(\mathrm{d}x) = \sum_{\emptyset \neq I \subset \{1,\ldots,d\}}(-1)^{|I|+1}\varphi\big( \norm{\bm{y}_I}_p\big),
\end{align*}
establishing the claim.
\end{proof}

The important aspects of the stochastic representation (\ref{shortref}) are that the random variables $\eta_k:=G_{\nu}^{-1}(\xi_1+\ldots+\xi_k)$ are independent of the random vectors $\bm{Z}^{(k)}$, and that the latter are bounded in the unit ball of the $\ell_p$-norm. Thus, this stochastic representation reflects the sequential integration (\ref{factor_ident}) in stochastic terms, with the $\eta_k$ accounting for the radial part and the $\bm{Z}^{(k)}$ for the directional (and in particular bounded) part. In the language of Poisson random measure, the random point measure
\begin{gather*}
P := \sum_{k \geq 1}\delta_{(\eta_k,\bm{Z}^{(k)})}
\end{gather*}
is a stochastic representation of a Poisson random measure on $(0,\infty] \times [0,1]^d$ with mean measure $\nu \times \IP(\bm{Z} \in \mathrm{d}\bm{z})$, and the points $\{\eta_k\}_{k \geq 1}$ denote a decreasing enumeration of the points of a Poisson random measure on $(0,\infty]$ with mean measure $\nu$. The fact that $\nu$ is non-finite implies that for arbitrary $\epsilon>0$ almost all $\eta_k$ lie within the interval $(0,\epsilon]$ almost surely, i.e.\ the $\eta_k$ tend to zero. Together with the boundedness of the $\bm{Z}^{(k)}$ this implies that the component-wise maxima in the stochastic representation (\ref{shortref}) are well-defined, since intuitively only the first few $k$ have non-negligible size. 
\par
Finally, it remains to be explained how to simulate random vectors $\bm{Y}$ with stochastic representation (\ref{stoch_maxid}) exactly, because it involves a maximum over infinitely many numbers. To this end, the decisive aspect is that the $\bm{Z}^{(k)}$ are bounded in $[\bm{0},\bm{1}]$, due to our decomposition into directional and radial part. This allows to generalize the algorithm of \cite{mai18} for the case $p=1$ to the general case $p \geq 1$, as we now explain. If we denote 
\begin{gather*}
M_n :=  \min_{j \in \{1,\ldots,d\}}\Big\{ \max_{k \in \{1,\ldots,n\}}\big\{ G_{\nu}^{-1}(\xi_1+\cdots+\xi_k)\,Z_j^{(k)} \big\} \Big\},\quad n \geq 1,
\end{gather*}
the $j$-th component of $\bm{Y}$ in (\ref{stoch_maxid}) is actually equal to
\begin{align*}
Y_j &=  \max_{k \geq 1}\big\{ G_{\nu}^{-1}(\xi_1+\cdots+\xi_k)\,Z_j^{(k)} \big\} =  \max_{k \in\{ 1,\ldots,N\}}\big\{ G_{\nu}^{-1}(\xi_1+\cdots+\xi_k)\,Z_j^{(k)} \big\},\\
N&=\min\{n \geq 1\,:\,G_{\nu}^{-1}(\xi_1+\ldots+\xi_{n+1}) \leq M_n\},
\end{align*}
and the random variable $N$ is independent of $j$ and almost surely finite, since $\bm{Z}$ is bounded. Summarizing, Algorithm \ref{algo_maxid} is an exact simulation algorithm for $\bm{Y}$, with \textproc{SimulateExp}$(n)$ denoting a sub-routine that generates $n$ iid standard exponentials.

\begin{algorithm}
\caption{Simulation of $\bm{Y}$ in (\ref{stoch_maxid}) with radial measure $\nu$}\label{algo_maxid}
\begin{algorithmic}[1]
\Procedure{SimulateY}{$p,d,\nu$}
\State $\bm{Y}=(Y_1,\ldots,Y_d) \gets (0,\ldots,0)$
\State $T \gets$\textproc{SimulateExp}$(1)$
\State $\eta \gets G_{\nu}^{-1}(T)$
\While{$\eta>\min\{Y_1,\ldots,Y_d\}$}
\State $\bm{\xi}=(\xi_1,\ldots,\xi_d) \gets$\textproc{SimulateExp}$(d)$
\State $V_p \gets $\textproc{SimulateVp}$(p,d)$
\State $\bm{Z} =(Z_1,\ldots,Z_d) \gets V_p\,\Big(\frac{\bm{\xi}}{\xi_1+\ldots+\xi_d}\Big)^{\theta}$
\For{$j=1,\ldots,d$}
\State $Y_j \gets \max\big\{ Y_j,\,\eta\,Z_j\big\}$
\EndFor
\State $T \gets T+$\textproc{SimulateExp}$(1)$
\State $\eta \gets G_{\nu}^{-1}(T)$
\EndWhile
\State \Return $\bm{Y}$
\EndProcedure
\end{algorithmic}
\end{algorithm}

\begin{example}[The negative logistic model]
We consider $\bm{Y}$ with distribution function equal to $\bm{y} \mapsto \exp\big(-f_{p}(1/y_1,\ldots,1/y_d)\big)$, where
\begin{gather*}
f_{p}(\bm{y}):=\sum_{j=1}^{d}(-1)^{j+1}\sum_{1 \leq i_1<\ldots<i_j \leq d}\Big( \sum_{k=1}^{d}y_k^{-p}\Big)^{-\theta},\quad \bm{y}\in [0,\infty)^d.
\end{gather*}
This is the so-called {negative logistic model}, the associated copula being termed {Galambos copula}, named after \cite{galambos75}. Notice that the law of $\bm{Y}^{p}$ equals a max-infinitely divisible distribution whose exponent measure has $\ell_1$-norm symmetric survival function generated by $\varphi(x)=x^{-\theta}$, as pointed out by Genest et al.\ \cite{genest18}. Two different exact simulation algorithms for $\bm{Y}$ can be found in \cite{dombry16}, and a third (truly different) one also in \cite{mai18}. For the case $p \geq 1$, Algorithm \ref{algo_maxid} is a distinct, original and exact simulation algorithm, which is based on the observation that the exponent measure of $\bm{Y}$ has an $\ell_p$-norm symmetric survival function generated by $\varphi(x)=1/x$. 
\end{example}

\begin{example}[An example with singular component] \label{ex_rAC1}
Consider the radial measure $\nu=\nu_{a}=a\,\sum_{k \geq 1}\delta_{1/k}$ for a parameter $a>0$. As pointed out in \cite[Example 2.3]{mai18} the associated generator $\varphi_{\nu}$ and required inverse $G^{-1}_{\nu}$ are 
\begin{align*}
\varphi_{\nu}(t) = a\,\sum_{k=1}^{\lfloor 1/t \rfloor}(1-k\,t)^{d-1},\quad G_{\nu}^{-1}(t) = 1/\Big\lceil \frac{t}{a} \Big\rceil.
\end{align*}
The scatter plots in Fig.\ \ref{fig:rACsing} depict samples of $\exp(-\varphi(\bm{Y}))$ for $d=2$, illustrating that $\bm{Y}$ is not absolutely continuous, and demonstrating the effect of introducing $p$ in comparison to the known case $p=1$.
\end{example}

\begin{figure}[h]
\centering
\includegraphics[width=0.4\linewidth]{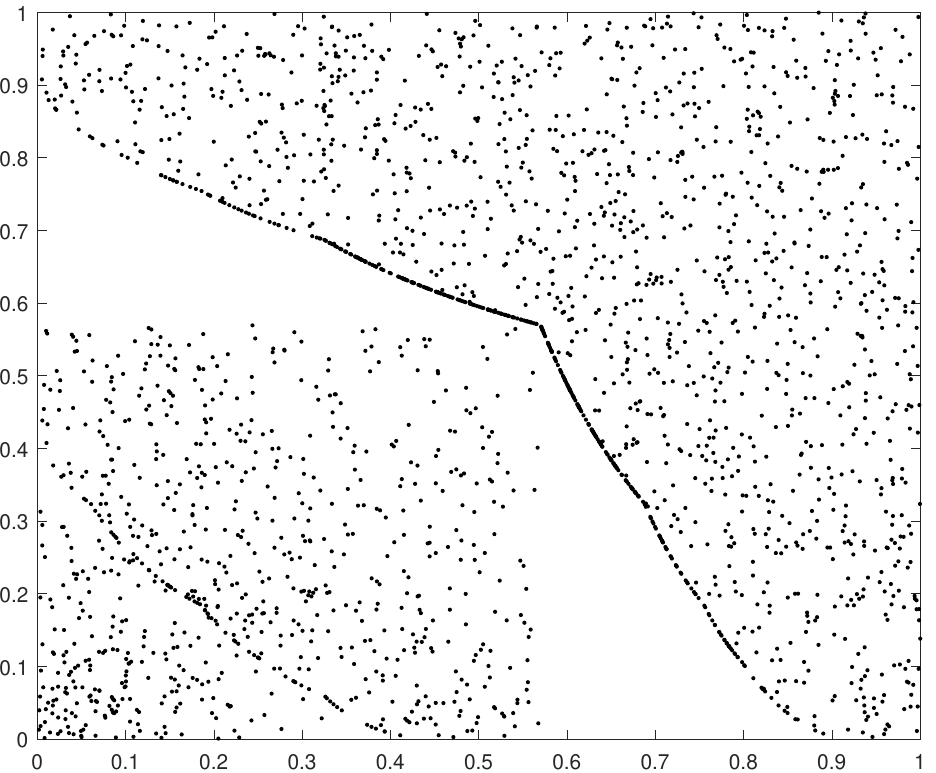}
\hfill
\includegraphics[width=0.4\linewidth]{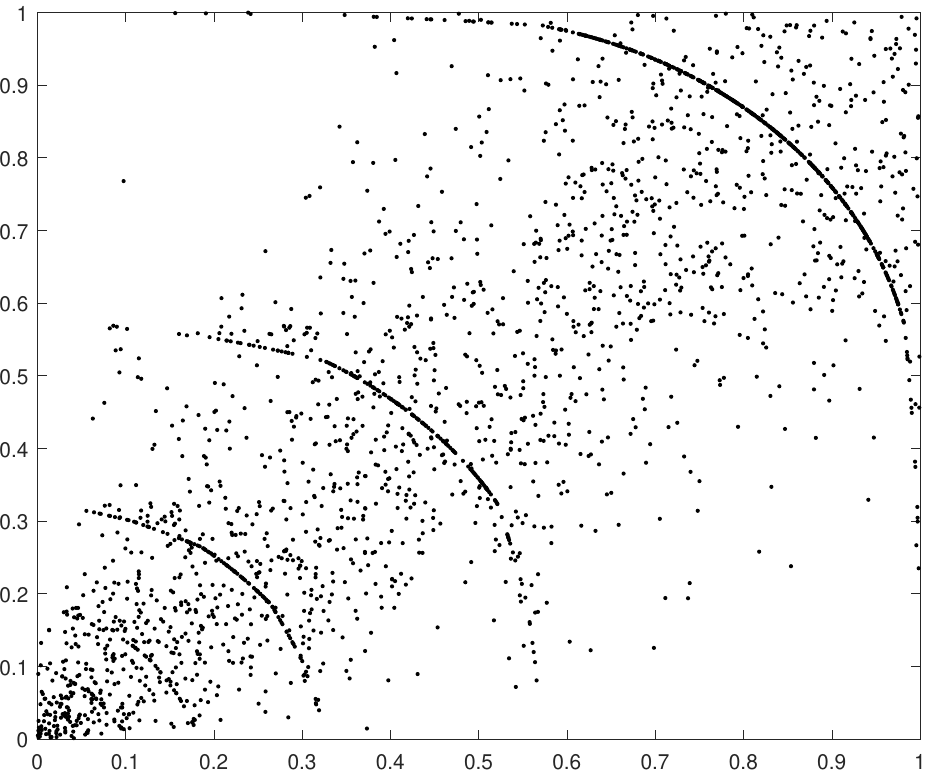}
\caption{Scatter plots from the reciprocal Archimedean copula in Example \ref{ex_rAC1}. Left: $p=1$ (so regular reciprocal Archimedean copula) and $a=1.125$. Right: $p=4$ and $a=1.125$.}
\label{fig:rACsing}
\end{figure}

\section*{Acknowledgments}
We thank the editor, the associate editor, and the anonymous referee for their valuable remarks on earlier versions of this manuscript. Ruodu Wang acknowledges financial support from the Natural Sciences and Engineering Research Council of Canada (RGPIN-2018-03823, RGPAS-2018-522590).

\appendix
\section{Some technical details}
\begin{lemma}[$\ell_p$-norm symmetric density]\label{lemma_lpcharsym}
The density $f$ of an absolutely continuous random vector $\bm{X}$ on $(0,\infty)^d$ is $\ell_p$-norm symmetric if and only if $\bm{X} \sim R\,\bm{U}^{(p)}$, where $\bm{U}^{(p)}$ is uniform on the $\ell_p$-sphere (restricted to the positive orthant w.l.o.g.), which we denote $S_{d,p}$, and $R$ is an independent positive and absolutely continuous random variable.
\end{lemma}
\begin{proof}
Clearly, if $\bm{X} \sim R\,\bm{U}^{(p)}$ then the density is $\ell_p$-norm symmetric. Now assume that $f(\bm{x})=g(\norm{\bm{x}}_p)$ for some function $g$ {of} one variable. We slightly generalize the computation on page 78 in \cite{mai17}, considering the mapping $h:(0,\infty)^d \rightarrow S_{d,p} \times (0,\infty)$, $\bm{x} \mapsto (\bm{x}/\norm{\bm{x}}_p,\norm{\bm{x}}_p)$. We observe that $|(h^{-1})^{'}(\bm{y},s)|=p\,s^{p+d-2}$ is independent of the first $d-1$ components of $h^{-1}$. For an arbitrary bounded and continuous function $b$ multivariate change of variables implies
\begin{align*}
\IE[b(\bm{X})] = \int b(\bm{x})\,f(\bm{x})\,\mathrm{d}\bm{x} = \int_{(0,\infty)}\int_{S_{d,p}} b(s\,\bm{y})\,\mathrm{d}\bm{y}\,g(s)\,p\,s^{p+d-2}\,\mathrm{d}s=  c_p\,\int_{(0,\infty)}\IE[b(R\,\bm{U}^{(p)})\,|\,R=s]\,g(s)\,p\,s^{p+d-2}\,\mathrm{d}s,
\end{align*}
where the positive constant $c_p$ denotes the volume of $S_{d,p}$. This implies the claim. 
\end{proof}

Finally, we sketch a proof for Lemma \ref{lemma_ident}.
\begin{proof}
To verify the first identity,  one may first prove via induction and integration by parts that
\begin{gather}
\beta_{m,n}(x) =  (m+n-1)!\,\sum_{k=0}^{n-1}\frac{x^{m+k}\,(1-x)^{n-1-k}}{(m+k)!\,(n-1-k)!}.
\label{cdf_beta}
\end{gather}
Using (\ref{cdf_beta}), the first identity is readily established. To verify the second identity, we make use of the first in $(\ast)$ below and observe
\begin{align*}
\beta_{m,n-1}(x)-\beta_{m,n}(x)  &= \int_0^x \frac{(m+n-2)!}{(m-1)!\,(n-2)!}y^{m-1}\,(1-y)^{n-2}-\frac{(m+n-1)!}{(m-1)!\,(n-1)!}y^{m-1}\,(1-y)^{n-1}\,\mathrm{d}y\\
&  = \int_0^x \frac{(m+n-2)!}{(m-1)!\,(n-2)!}y^{m-1}\,(1-y)^{n-2}\,\Big( y-\frac{m}{n-1}\,(1-y)\Big)\,\mathrm{d}y\\
&  = \frac{m}{m+n-1}\,\big( \beta_{m+1,n-1}(x)-\beta_{m,n}(x)\big) \stackrel{(\ast)}{=} -\frac{m}{m+n-1}\,\binom{m+n-1}{m}\,x^{m}\,(1-x)^{n-1}\\
&  = -\binom{m+n-2}{m-1}\,x^m\,(1-x)^{n-1}.
\end{align*}
\end{proof}

\end{document}